\documentclass[11pt,reqno]{amsart}

\usepackage{fullpage}
\usepackage{amssymb,latexsym}
\usepackage{mathtools}
\usepackage{dsfont}
\usepackage{graphicx}
\usepackage{enumitem}
\usepackage[british]{babel}
\usepackage[babel]{microtype}          % Improves spacing
\usepackage{verbatim}
\usepackage{xcolor}

\usepackage[numbers,sort]{natbib}

\makeatletter
\def\NAT@spacechar{~}% NEW
\makeatother

\usepackage{hyperref}
\usepackage[capitalise]{cleveref}  % For improved referencing options
\hypersetup{colorlinks=true,
   citecolor=blue,
   filecolor=blue,
   linkcolor=blue,
   urlcolor=blue
}
\usepackage{url}

\numberwithin{equation}{section}
\newtheorem{theorem}{Theorem}[section]
\newtheorem{lemma}[theorem]{Lemma}
\newtheorem{definition}[theorem]{Definition}
\newtheorem{fact}[theorem]{Fact}
\newtheorem{proposition}[theorem]{Proposition}
\newtheorem{corollary}[theorem]{Corollary}
\newtheorem{claim}[theorem]{Claim}

\newtheorem{observation}[theorem]{Observation}

\newcommand{\cCrsn}{\cC_{r,s,n}}
\newcommand{\cCrtn}{\cC_{r,2,n}}

\newcommand{\Krsn}{K_{r,s,n}}
\newcommand{\Krtn}{K_{r,2,n}}
\newcommand{\EE}{\mathbb{E}}
\newcommand{\NN}{\mathbb{N}}
\newcommand{\PP}{\mathbb{P}}
\newcommand{\ZZ}{\mathbb{Z}}
\newcommand{\Gnp}{G(n,p)}
\newcommand{\Gnm}{G(n,m)}

\newcommand{\Ctwo}{C^{e}_{4,n}}
\newcommand{\cC}{\mathcal{C}}
\newcommand{\cF}{\mathcal{F}}
\newcommand{\eps}{\varepsilon}
\newcommand{\cH}{\mathcal{H}}

\newcommand{\var}{\mathrm{Var}}

\newcommand{\COMMENT}[1]{}

\title{Spanning \texorpdfstring{$F$}{F}-cycles in random graphs}
\author{Alberto Espuny Díaz}
\author{Yury Person}
\address{Institut für Mathematik, Technische Universität Ilmenau, 98684 Ilmenau, Germany}
\email{alberto.espuny-diaz\,|\,yury.person@tu-ilmenau.de}
\thanks{Research is supported by the Carl Zeiss Foundation and by DFG grant PE 2299/3-1.}

\begin{document}

\date{\today}

\begin{abstract}
We extend a recent argument of Kahn, Narayanan and Park (\emph{Proceedings of the AMS}, to appear) about the threshold for the appearance of the square of a Hamilton cycle to other spanning structures.
In particular, for any spanning graph, we give a sufficient condition under which we may determine its threshold.
As an application, we find the threshold for a set of cyclically ordered copies of $C_4$ that span the entire vertex set, so that any two consecutive copies overlap in exactly one edge and all overlapping edges are disjoint.
This answers a question of Frieze.
We also determine the threshold for edge-overlapping spanning $K_r$-cycles.
\end{abstract}
\maketitle

\section{Introduction}\label{sec:intro}

The study of threshold functions for the appearance of spanning structures plays an important role in the theory of random graphs.
Unlike in the case of small subgraphs, which was resolved by \citet{ErdRen60} (for balanced graphs) and by \citet{Bol81} (for general graphs), in the case of general spanning structures only sufficient conditions are known.
These lead to upper bounds for the threshold of a general spanning graph, although the expectation threshold conjecture of \citet{KK07}, if true, predicts the threshold for \emph{any} graph up to a logarithmic factor.
 
Apart from particular structures where the thresholds are known, such as perfect matchings~\cite{ErdRen66}, $F$-factors~\cite{JohKahVu08}, Hamilton cycles~\cite{Kor77,Pos76} or spanning trees~\cite{Montgomery19} (to name a few), the most general result providing upper bounds was, until recently, due to \citet{Ri00}, giving in some cases asymptotically optimal upper bounds (lattices, hypercubes, $k$-th powers of Hamilton cycles for $k\ge 3$~\cite{KO12}).  An excellent survey by \citet{Boe17} provides references to many other results, in particular algorithmic ones.
  
The recent breakthrough work by \citet{FKNP19} established the fractional expectation threshold conjecture of \citet{Tal10}, providing in many cases optimal thresholds or being off by at most a logarithmic factor.
The subsequent work by \citet{KNP20} exploited the proof approach in~\cite{FKNP19} in a more efficient way, allowing to erase the logarithmic factor in the case of the square of a Hamilton cycle, and thus proving the threshold for its appearance to be $n^{-1/2}$. 

In a recent paper, \citet{Frieze20} studied thresholds for the containment of spanning $K_r$-cycles, i.e., cyclically ordered edge-disjoint copies of $K_r$ with two consecutive copies sharing a vertex.
He proved the optimal threshold of the form $n^{-2/r}\log^{1/\binom{r}{2}}n$ by reducing this problem to another result of Riordan about coupling the random graph with the random $r$-uniform hypergraph~\cite{Ri18} (see also the work of \citet{Hec18} for the triangle case).
Frieze also raised  the question about the threshold for the containment of a spanning $C_4$-cycle, where the copies of $C_4$ are ordered cyclically and two consecutive cycles overlap in exactly one edge, whereby each cycle $C_4$ overlaps with two copies of $C_4$ in opposite edges (there are some possible variations, but this would be a canonically defined structure).
Such $C_4$-cycles are referred to in~\cite{Frieze20} as a $C_4$-cycle with overlap $2$, where it is also observed that the threshold for its appearance is at most $ n^{-2/3} \log n$, which follows from~\cite{FKNP19}. 
 
The purpose of this paper is to contribute to the large body of work on thresholds for spanning structures by  establishing thresholds for spanning $2$-overlapping $C_4$-cycles (which we denote by $\Ctwo$), thus answering the question of Frieze~\cite{Frieze20}, and also for $2$-overlapping $K_r$-cycles (defined below) for $r\geq4$.
Both structures cannot be handled directly by the results in~\cite{FKNP19, Ri00}. 
In order to obtain these results, we generalise the approach of \citet{KNP20}.
As the results, we establish the following thresholds.

The first theorem answers the question of Frieze~\cite{Frieze20}. 
\begin{theorem}\label{thm:Ctwo-threshold}
The threshold for the appearance  of $\Ctwo$ in $G(2n,p)$ is $\Theta(n^{-2/3})$.
\end{theorem}

Our second result generalises the recent work of \citet{KNP20} on the threshold for the square of a Hamilton cycle.
The square of a Hamilton cycle can be seen as the particular case $r=3$ of a structure which we call $2$-overlapping (or edge-overlapping) spanning $K_r$-cycle and denote by $K_{r,2,n}$, for $r\geq3$.
This consists of a set of cyclically ordered copies of $K_r$, where consecutive cliques share exactly one edge and, if $r\geq4$, all other cliques are pairwise vertex-disjoint\COMMENT{In the case $r=3$, this is impossible, and we enforce that each clique shares exactly one vertex with the consecutive of its consecutive; this precisely defines the square of a Hamilton cycle.}.

\begin{theorem}\label{thm:Kr-cycle}
Let $r\ge 3$ and $n\in \NN$ with $(r-2)\mid n$.
Then, the threshold for the appearance of $K_{r,2,n}$ in $\Gnp$ is $\Theta(n^{-2/(r+1)})$.
\end{theorem}

To prove \cref{thm:Ctwo-threshold,thm:Kr-cycle}, we state and prove a general lemma (the fragmentation lemma, \cref{lem:fragmentation}), which has potential to handle more spanning structures.
This lemma is a generalisation of the work of Kahn, Narayanan and Park on the square of a Hamilton cycle~\cite[Lemma~3.1]{KNP20} to handle structures for which constantly many rounds of exposure may be necessary, in contrast to~\cite{KNP20}, where only two rounds are used, and to~\cite{FKNP19}, where logarithmically many rounds are necessary.

The organisation of the paper is as follows.
In the next section, \cref{sec:fragmentation}, we provide the main definitions, state a general lemma (the fragmentation lemma, \cref{lem:fragmentation}), and use it to establish a general theorem (\cref{thm:main}) about thresholds for certain spanning graphs.
\Cref{thm:main} is actually the main general result of the paper, and \cref{thm:Ctwo-threshold,thm:Kr-cycle} are two of its applications.
We prove these two applications in \cref{sec:special_cases}. 
Finally, in \cref{sec:conclude} we collect a few remarks, and in the Appendix we provide the proof of \cref{lem:fragmentation}.

\section{A general theorem for thresholds}\label{sec:fragmentation}

Given any real numbers $a$ and $b$, we write $[a,b]$ to refer to the set $\{n\in\mathbb{Z}:a\leq n\leq b\}$.
For an integer $n$, we often abbreviate $[n]\coloneqq[1,n]$.
We use standard $O$ notation for asymptotic statements.

A hypergraph $\cH$ on the vertex set $V\coloneqq V(\cH)$ is a subset of the power set $2^V$.
The elements of $\cH$ are referred to as edges.
The hypergraph $\cH$ is said to be $r$-bounded if all its edges have cardinality at most $r$, and $r$-uniform if all the edges have exactly $r$ vertices.
Oftentimes, we will consider multihypergraphs $\cH$ on $V$, where we view $\cH$ as a multiset with elements from $2^V$. 
To ease readability, we will often refer to multihypergraphs as hypergraphs.
We also omit floor and ceiling signs whenever they do not affect our asymptotic computations.

Following \cite{KNP20}, we say that a (multi-)hypergraph $\cH$ is \emph{$q$-spread} if, for every $I\subseteq V(\cH)$, we have
\begin{equation*}\label{eq:spread-def}
|\cH\cap \langle I\rangle|\le q^{|I|}|\cH|,
\end{equation*}
where $\langle I\rangle\coloneqq \{J\subseteq V(\cH):I\subseteq J\}$ and $\cH\cap \langle I\rangle$ is the set of edges  of $\cH$ in $\langle I\rangle$ (with multiplicities if $\cH$ is a multihypergraph).
The \emph{spreadness} of $\cH$ is the minimum $q$ such that $\cH$ is $q$-spread.

Let $S\in \cH$ and $X\subseteq V(\cH)$.
For any $J\in \cH$ such that $J\subseteq S\cup X$, we call the set $J\setminus X$ an \emph{$(S, X)$-fragment}.
Given some $k\in \NN$, we say that the pair $(S,X)$ is \emph{$k$-good} if some $(S,X)$-fragment has size at most $k$, and we say it is \emph{$k$-bad} otherwise.

More generally, let $\cH_0$ be some $k_0$-bounded (multi-)hypergraph. 
Let $k_0\geq k_1\geq \ldots\geq k_t$ be a sequence of integers and $X_1,\ldots,X_t$ be a sequence of subsets of $V(\cH_0)$.
Then, we define a sequence of $k_i$-bounded multihypergraphs $\cH_1,\ldots,\cH_{t}$ inductively as follows.
Let $i\in[t]$, and assume the hypergraph $\cH_{i-1}$ is already defined.
Then, consider each $S\in \cH_{i-1}$ such that $(S,X_i)$ is a $k_i$-good pair, and let $\cH_i$ be the multihypergraph which consists of one (arbitrary) $(S,X_i)$-fragment of size at most $k_i$ for each such $k_i$-good pair $(S,X_i)$.
That is, we define $\mathcal{G}_i\coloneqq\{S\in\cH_{i-1}:(S,X_i)\text{ is }k_i\text{-good}\}$ and, for each $S\in\mathcal{G}_i$, $\mathcal{J}_i(S)\coloneqq\{J\setminus X_i:J\in\cH_{i-1},J\subseteq S\cup X_i, |J\setminus X_i|\leq k_i\}$.
We then fix an arbitrary function $f_i\colon\mathcal{G}_i\to\bigcup_{S\in\mathcal{G}_i}\mathcal{J}_i(S)$ such that $f_i(S)\in\mathcal{J}_i(S)$ for every $S\in\mathcal{G}_i$ (for instance, we may simply pick the lexicographically smallest element in the set) and define
\[
\cH_{i}\coloneqq \{f_i(S):S\in\mathcal{G}_i\}.
\]
We will refer to the sequence $(\cH_0,\cH_1,\ldots,\cH_t)$ as a \emph{fragmentation process} with respect to $(k_1,\ldots, k_t)$ and $(X_1,\ldots,X_t)$. 
In our applications, we will let $X_1,\ldots,X_t$ be random subsets of $V(\cH_0)$ and choose a suitable sequence $k_0,\ldots,k_t$ which will guarantee that the hypergraphs in the sequence do not become very small (with high probability).
Observe that the fragments at the $i$-th step of this process (that is, the edges of $\cH_i$) correspond to subsets of the edges of $\cH_0$ which have not been covered by the sets $X_1,\ldots,X_i$.
In particular, for all $i\in[t]$ and all $I\subseteq V(\cH_0)$ we have that\COMMENT{Let $S\in\cH_i\cap\langle I\rangle$, so in particular $S\in\cH_i$. But then, as $\cH_i$ was obtained by a fragmentation process, $S$ is a subset of some uniquely defined $S'\in\cH_0$ .
Since $\langle I\rangle$ contains all supersets of $I$, we have that $S$ is a superset of $I$ and, thus, so is $S'$. Hence $S'\in\cH_0\cap\langle I\rangle$.\\
Note about $S'$: say that $i=1$ (for larger $i$, the same holds in an iterative way). Then, $S'$ is \emph{not} chosen so that $S=S'\setminus X_1$, but rather as the set $S'$ from which we have chosen $S$ as an $(S',X_1)$-fragment. This guarantees that, for each $S\in\cH_i$, the chosen $S'$ is distinct (possibly as an element of the multiset), and thus the above really yields the desired bound. The reason why we are still guaranteed that $S\subseteq S'$ is that $S=J\setminus X_1$ for some $J\subseteq S'\cup X_1$, so $S=J\setminus X_1\subseteq S'\setminus X_1\subseteq S'$.}
\begin{equation}\label{equa:fragmentationProperty}
    |\cH_i\cap \langle I\rangle|\leq|\cH_0\cap \langle I\rangle|.
\end{equation}

While the general framework developed in \cite{FKNP19,KNP20} works for arbitrary hypergraph thresholds, here we focus on graphs.
Let $F$ be some (possibly spanning) subgraph of the complete graph $K_n$, and let $\cF$ denote the set of all copies of $F$ in $K_n$\COMMENT{In more generality, we could define $\cF$ as the subgraph of all minimal elements of an increasing property $\mathcal{P}$, in the same way as in \cite{FKNP19}; I believe our methods would transfer as long as each element of $\cF$ has the same size, so $\cF$ is uniform; they should also work if it is not uniform, but we might need to be a bit more careful.}.
We will identify copies of $F$ from $\cF$ with their edge sets, and we thus view $\cF$ as a $k$-uniform hypergraph, where $k=|E(F)|$, on the vertex set $M\coloneqq \binom{[n]}{2}$.

We now define a strengthening of the notion of spreadness of hypergraphs which is key for our results.
For $q,\alpha,\delta\in(0,1)$, we say that a $k$-bounded hypergraph $\cF$ on vertex set $M$ is \emph{$(q,\alpha,\delta)$-superpread} if it is $q$-spread and, for any $I\subseteq M$ with $|I|\le \delta k$, we have
\begin{equation*}\label{eq:superspread-def}
|\cF\cap \langle I\rangle|\le q^{|I|} k^{-\alpha c_I} |\cF|,
\end{equation*}
where $c_I$ is the number of components of $I$ (when $I$ is viewed as a subgraph of $K_n$).
The role of the term $k^{-\alpha c_I}$ will become clear later, but, roughly speaking, it will be responsible for bounding the threshold by $O(q/\alpha)$.
The value of the constant $\delta$ actually plays no role in the result, but we do need it to be bounded away from $0$ for our approach to work.

The following result is the main lemma of the paper.
It will be used to iteratively build a spanning copy of $F$ in $\Gnp$ through a fragmentation process.

\begin{lemma}\label{lem:fragmentation}
Let $d,\alpha,\delta>0$ with $\alpha,\delta<1$.
Then, there is a fixed constant $C_0$ such that, for all $C\ge C_0$ and $n\in\NN$, the following holds.
Let $F$ be some subgraph of $K_n$ with $\Delta(F)\le d$\COMMENT{Note: this condition cannot be relaxed because we use \cref{lem:num_subgraphs}.} and $k_0\coloneqq|E(F)|=\omega(1)$, and let $\cF$ be the set of all copies of $F$ in $K_n$.
Assume that $\cF$ is $(q,\alpha,\delta)$-superspread with $q\geq4k_0/(Cn^2)$ and that $(\cH_0,\cH_1,\ldots,\cH_i)$ is some fragmentation process with $\cH_0\coloneqq \cF$ such that, for each $j\in[i]$, $\cH_j$ is $k_j$-bounded and $|\cH_j|\ge |\cH_{j-1}|/2$, and $k_i=\omega(k_0^{\alpha})$.
Then, for $w\coloneqq Cq \binom{n}{2}$, $k\coloneqq k_ik_0^{-\alpha}$ and $X$ chosen uniformly at random from $\binom{M}{w}$, we have
\begin{equation}\label{eq:fragmentation}
\EE\left[\left\lvert\left\{ (S,X) : S\in \cH_i, (S,X)\text{ is }k\text{-bad}\right\}\right\rvert\right]\le 2C^{-k/3} |\cH_i|.
\end{equation}
\end{lemma}

The proof of \cref{lem:fragmentation} closely follows the proofs of Lemma~3.1 from~\cite{KNP20} and Lemma~3.1 from~\cite{FKNP19}.
Therefore, for the sake of completeness, we give its proof in \cref{app:Fragmentation}, for the convenience of the interested reader.

Equipped with \cref{lem:fragmentation} we can now establish the following.

\begin{theorem}\label{thm:main}
Let $d,\alpha,\delta,\eps>0$ with $\alpha,\delta<1$.
Then, there is a fixed constant $C_0$ such that, for all $C\ge C_0$ and $n\in\NN$, the following holds.
If $F$ is a subgraph of $K_n$ with $\Delta(F)\le d$ and $k_0\coloneqq|E(F)|=\omega(1)$ and the hypergraph $\cF$ of all copies of $F$ is $(q,\alpha,\delta)$-superspread with $q\geq4k_0/(Cn^2)$, then, for $p\ge C q$,
\[
\PP\left[F\subseteq \Gnp\right]\geq1-\eps.
\]
\end{theorem}

This result immediately provides an upper bound of $Cq$ for the threshold for the appearance of $F$ as a subgraph of $G(n,p)$\COMMENT{We do not need to use the results of Friedgut: the original result of Bollobás and Thomason (see the proof in Frieze-Karonski) already shows that $p$ as above is an upper bound on the threshold. Perhaps we should cite some of these.}.
If a matching lower bound can be found (say, by the standard first moment method\COMMENT{Can we always show that the spreadness of a hypergraph is a lower bound for the threshold?}), then this establishes the threshold for the appearance of any graph $F$ which satisfies the conditions in the statement.

The proof of \cref{thm:main} follows along similar lines as the proofs in~\cite{FKNP19,KNP20}: one proceeds in rounds of sprinkling random edges by showing that, after each round of exposure (which corresponds to a step of the fragmentation process), the random graph contains larger pieces of the desired structure (or, conversely, the missing fragments become smaller).
In the general proof in~\cite{FKNP19}, the authors show that the progress in each round shrinks the percentage of the edges from the desired structure by a factor of $0.9$, which results in logarithmically many steps and, thus, a $\log n$ factor with respect to the fractional expectation threshold of the structure (this result is quite general, though, and oftentimes a logarithmic factor is indeed needed, as in the case of spanning trees, Hamilton cycles and $K_r$-factors in random graphs, or of perfect matchings and loose Hamilton cycles in random hypergraphs).
The threshold for the square of a Hamilton cycle $K_{3,2,n}$ happens to be $n^{-1/2}$, and in this case, as shown in~\cite{KNP20}, two rounds suffice: the shrinkage factor there is $n^{-1/2}$, so that after the first round a second moment computation suffices in the second round of exposure/sprinkling.
We show that the threshold for the appearance of $F$ is at most $O(q)$ and for this we will need $1/\alpha$ rounds (exposing each time edges with probability $Cq$, for some constant $C$): the shrinkage factor in all but the last round will be $n^{-\alpha}$, so that we can apply the second moment method in the last round.

In the proof of \cref{thm:main} we make use of the following auxiliary lemma.
Again, its proof follows similarly as the proof of Proposition~2.2 in~\cite{KNP20}, and we thus defer it to \cref{app:Fragmentation}.

\begin{lemma}\label{lem:num_subgraphs}
Let $F$ be a graph with $f$ edges and maximum degree $d$.
Then, the number of subgraphs $I$ of $F$ with $\ell$ edges and $c$ components is at most
\[
(4ed)^{\ell}\binom{f}{c}.
\]
\end{lemma}

\begin{proof}[Proof of \cref{thm:main}]
We first note that, by adjusting the value of $C_0$, we may assume that $n$ is sufficiently large, and therefore $k_0$ is sufficiently large too.
We may also assume that $q<C_0^{-1}$.
In the beginning, we will switch and work with the $\Gnm$ model instead of $\Gnp$.
This can be done easily since these models are essentially equivalent for $m=p\binom{n}{2}$ (see, e.g.,~\cite[Proposition~1.12]{JLR00}).

We proceed as follows.
We consider $G(n,m_1)\cup G(n,m_2)\cup\ldots\cup G(n,m_t)$ with $t=\lceil1/\alpha\rceil-1$ and $m_i=Kq\binom{n}{2}$ for each $i\in[t]$, where $K$ is assumed to be sufficiently large throughout (and $C_0$ will be defined as $2(t+1)K$).
We then define a fragmentation process on $\mathcal{F}$ with respect to $(k_1,\ldots,k_t)$ and $(G(n,m_1),\ldots,G(n,m_t))$, where the integers $k_1,\ldots,k_t$ will be defined shortly.
We prove that a.a.s.~each step of this fragmentation process satisfies the conditions of \cref{lem:fragmentation}, so that we may iteratively apply it and conclude that each of the subsequent hypergraphs is not `too small'.
At the end of this process, we will be sufficiently `close' to a copy of $F$ that a second moment argument will yield the result.

To be precise, we first consider the hypergraph $\cH_0\coloneqq \cF$ and take $X_1\coloneqq G(n,m_1)$ and $k_1\coloneqq k_0^{1-\alpha}$.
We consider a first step in the fragmentation process.
We obtain a multihypergraph $\cH_1$ of $(S,X_1)$-fragments which is $k_1$-bounded, where each $S$ is an edge of $\mathcal{H}_0$.
In particular, by the assertion~\eqref{eq:fragmentation} of \cref{lem:fragmentation} and Markov's inequality, we have that
\begin{equation}\label{eq:successful}
\PP\left[|\cH_1|\ge |\cH_0|/2\right]\ge 1-4K^{-k_1/3}. 
\end{equation}

Suppose now that we have already run the fragmentation process $(\cH_0,\cH_1,\ldots,\cH_i)$, for some $i\in[t-1]$, and that $|\cH_j|\ge |\cH_{j-1}|/2$ for all $j\in[i]$.
We run one further step of the fragmentation process with $X_{i+1}\coloneqq G(n,m_{i+1})$ and $k_{i+1}\coloneqq k_ik_0^{-\alpha}$ to obtain a $k_{i+1}$-bounded hypergraph $\cH_{i+1}$ of $(S,X_1\cup\ldots\cup X_{i+1})$-fragments (where, again, each $S$ is an edge of $\mathcal{H}_0$).
By another application of \cref{lem:fragmentation} and Markov's inequality, we obtain that
\begin{equation}\label{eq:successful2}
\PP\left[|\cH_{i+1}|\ge |\cH_i|/2\right]\ge 1-4K^{-k_{i+1}/3}.
\end{equation}

We say that the fragmentation process $(\cH_0,\cH_1,\ldots,\cH_t)$ is \emph{successful} if $|\cH_j|\ge |\cH_{j-1}|/2$ for all $j\in[t]$.
Let $\beta\coloneqq1-t\alpha$, and note that, by the definition of $t$, we have $0<\beta\leq\alpha$.
By \eqref{eq:successful} and \eqref{eq:successful2}, we conclude that the probability that the fragmentation process $(\cH_0,\cH_1,\ldots,\cH_{t-1})$ which we run is successful is\COMMENT{Note we are assuming $k_0\to\infty$ and, since $\beta$ is a positive constant (but $\beta-\alpha\leq0$), we also have $k_0^\beta\to\infty$.
This is clearly the leading term, as in all other cases the exponent is a smaller constant.}
\[
1-4\sum_{i=1}^{t}K^{-k_i/3}\ge1-4\sum_{i=1}^{\lceil{1}/{\alpha}\rceil-1}K^{-k_0^{1-i\alpha}/3}=1-O\left(K^{-k_0^{\beta}/3}\right).
\]
To summarise, a.a.s.~the fragmentation process is successful and, thus, yields a $k_{t}$-bounded multihypergraph $\cH_t$ of $(S,X_1\cup\ldots\cup X_t)$-fragments, where $k_{t}=k_0^{\beta}$, $|\cH_t|\geq2^{-t}|\cH_0|$ and each $S$ is an element of $\cH_0$. 

We now apply one more round of sprinkling.
In this final round we switch and work with the random set $X\coloneqq G(n,p)$ with $p=Kq$.
We may also assume that $\cH_t$ is $k_t$-uniform, since every set $S\in\cH_t$ is contained in some $S'\in\cF$ and thus we can add some arbitrary $k_t-|S|$ vertices from $S'\setminus S$ to~$S$.
The proof now will proceed along the same lines as the proof in~\cite[Theorem~1.2]{KNP20}.

Define the random variable $Y\coloneqq |\{S\in\cH_t:S\subseteq G(n,p)\}|$.
Our aim is to estimate the variance of~$Y$ and to show that $\PP[Y=0]\le \eps$.
This would mean that the random graph $G(n,p)\cup\bigcup_{i=1}^{t} G(n,m_i)$ contains a copy of $F$ with probability at least $1-\eps$ (by~\cite[Proposition~1.12]{JLR00}, this also applies to $G\left(n,C_0q\right)$).

We estimate the variance of $Y$ as follows (recall that we work in $G(n,p)$ now).
Let $R\in \cH_t$, so $|R|=k_t=k_0^{\beta}$.
Then, using the fact that $\cF$ is $(q,\alpha,\delta)$-superspread and \eqref{equa:fragmentationProperty}, for each $\ell\in[k_t]$ we have that
\begin{align*}
|\{S\in \cH_t: |S\cap R|=\ell\}|&\le \sum_{L\subseteq R, |L|=\ell} |\cH_t\cap\langle L\rangle|\le  \sum_{L\subseteq R, |L|=\ell}  |\cF\cap\langle L\rangle|
\le \sum_{L\subseteq R, |L|=\ell}  q^{|L|} k_0^{-\alpha c_L} |\cF|,\\
&=\sum_{c=1}^\ell \sum_{L\subseteq R, |L|=\ell, c_L=c}  q^{\ell} k_0^{-\alpha c} |\cF|
\overset{\text{\cref{lem:num_subgraphs}}}{\le} \sum_{c=1}^\ell (4ed)^\ell \binom{k_t}{c} q^{\ell} k_0^{-\alpha c} |\cF|\\
&=q^{\ell} |\cF| (4ed)^\ell \sum_{c=1}^\ell  \binom{k_t}{c} k_0^{-\alpha c} \le q^{\ell} |\cF| (4ed)^\ell \sum_{c=1}^\ell  \left(\frac{ ek_t k_0^{-\alpha}}{c}\right)^c
=q^{\ell} |\cF| e^{O(\ell)},
\end{align*}
where the implicit constants in the $O$ notation are independent of $K$.
We therefore get the following bound on the variance:
\[
\var[Y]\le p^{2k_t}\sum_{R,S\in \cH_t,\, R\cap S\neq\varnothing} p^{-|R\cap S|}\overset{|\cH_t|\le |\cF|}{\le} |\cF|^2p^{2k_t}\sum_{\ell=1}^{k_t} e^{O(\ell)}p^{-\ell}q^{\ell}=O\left(\EE[Y]^2/K\right),
\]
where we use the facts that $p=Kq$, $\EE[Y]=p^{k_t}|\cH_t|$ and $|\cH_t|\ge 2^{-t}|\cF|=\Theta(|\cF|)$.
By letting $K$ be sufficiently large, the result follows by Chebyshev's inequality.
\end{proof}

\section{Applications of Theorem~\ref{thm:main}}\label{sec:special_cases}

We use this section to prove \cref{thm:Ctwo-threshold,thm:Kr-cycle} as applications of \cref{thm:main}.

\subsection{Spanning \texorpdfstring{$C_4$}{C4}-cycles}\label{sect31}

Throughout this section, we assume that $n$ is even.
Observe that the graph $\Ctwo$ has $3n/2$ edges.
Let $\cC$ be the $(3n/2)$-uniform hypergraph on the vertex set $M=\binom{[n]}{2}$ where we see (the set of edges of) each copy of $\Ctwo$ as an edge of $\cC$.
We write $|\cC|$ for the number of its edges and notice that $|\cC|=(n-1)!/2$.
Indeed, consider an arbitrary labelling $v_1,\ldots,v_n$ of the vertices.
We define a copy of $\Ctwo$ uniquely based on this ordering: we consider a matching between the set of even vertices and odd vertices (where we add the edge $v_{2i-1}v_{2i}$ for each $i\in[n/2]$), and then we define a cycle of length $n/2$ on the set of even vertices and another cycle in the set of odd vertices (where the edges of these cycles join the vertices which are closest in the labelling, seen cyclically).
In this way, each of the $C_4$'s which conform the copy of $\Ctwo$ is given by four consecutive vertices in the labelling, starting with an odd vertex $v_{2i-1}$, so that its edges are $\{v_{2i-1}v_{2i},v_{2i+1}v_{2i+2},v_{2i-1}v_{2i+1},v_{2i}v_{2i+2}\}$.
Now one can easily verify that there are $2n$ different labellings which yield the same copy of $\Ctwo$ (there are $n/2$ possible starting points while maintaining the same cyclic ordering; if the ordering is reversed, the resulting graph is the same; and if all pairs of vertices $\{v_{2i-1}v_{2i}\}$ are swapped, the resulting graph is also the same).

Recall that a hypergraph $\cC$ is $q$-spread, for some $q\in(0,1)$, if $|\cC\cap \langle I\rangle|\le q^{|I|}|\cC|$ for all $I\subseteq M$, where $\langle I\rangle$ denotes the set of all supersets of $I$. 
Moreover, for $\alpha,\delta\in(0,1)$, we say $\cC$ is $(q,\alpha,\delta)$-superpread if it is $q$-spread and, for every $I\subseteq M$ with $|I|\le 3\delta n/2$, we have
\begin{equation*}
|\cC\cap \langle I\rangle|\le q^{|I|} (3n/2)^{-\alpha c_I} |\cC|,
\end{equation*}
where $c_I$ is the number of components of $I$.
Our main goal now is to establish that the hypergraph $\cC$ is $(250n^{-2/3},1/3,1/15)$-superspread.

\begin{lemma}\label{obs:small-edge-spread}
Let $I\subseteq\Ctwo$ be a graph with $\ell\le n/10$ edges and $c$ components.
Then, we have
\[
|V(I)|-c\geq\frac{2}{3}\ell+\frac{c}{3}.
\]
\end{lemma}

\begin{proof}
Let $I_1,\ldots,I_c$ be the components of $I$ with at least one edge, and let $v_1,\ldots,v_c$ be the number of vertices spanned by $I_1,\ldots,I_c$, respectively. 
Since for all $j\in[c]$ we have $|I_j|\le|I|\le n/10$, we conclude the following easy bound on any component: 
\begin{equation}\label{eq:edge-bound}
|I_j|\le \frac{1}{2}\left(4\cdot 2+(v_j-4)\cdot 3\right)=\frac{3}{2}v_j-2.
\end{equation}
Indeed, this holds since the maximum degree of $I$ is at most $3$ and in every component $I_j$ with $4\leq v_j\leq n/10$ there are four vertices whose sum of degrees is at most $8$.
For $v_j\in[2,3]$ we have $|I_j|= v_j-1$, hence the bound given in~\eqref{eq:edge-bound} holds in these cases as well.
Summing over all $j\in[c]$, we obtain that \[\ell=|I|\leq\frac32|V(I)|-2c,\]
which yields the desired result by rearranging the terms.
\end{proof}

\begin{lemma}\label{obs:large-edge-spread}
Let $I\subseteq\Ctwo$ be a graph with $\ell$ edges and $c$ components.
Then, we have
\[
|V(I)|-c\geq\frac{2}{3}\ell-1.
\]
\end{lemma}

\begin{proof}
Since every vertex of $\Ctwo$ has degree $3$, the bound in the statement holds trivially if $I$ has only one component.
We may thus assume that $I$ contains at least two components with at least one edge each.
But then, one can directly check that \eqref{eq:edge-bound} must hold, and we can argue exactly as in \cref{obs:small-edge-spread}, which leads to a better bound than claimed in the statement.
\end{proof}

\begin{lemma}\label{obs:general-spread}
The hypergraph $\cC$ is $(250n^{-2/3},1/3,1/15)$-superspread.
\end{lemma}

\begin{proof}
Let $I\subseteq M$.
We need to obtain upper bounds for $|\cC\cap \langle I\rangle|$.
If $I$ is not contained in any copy of $\Ctwo$, then $|\cC\cap \langle I\rangle|=0$, so we may assume $I$ is a subgraph of some copy of $\Ctwo$.
Recall that a copy of $\Ctwo$ can be defined by an ordering of $[n]$ and that exactly $2n$ such orderings define the same copy of $\Ctwo$. 
Thus, it suffices to bound the number of orderings of $[n]$ which define a copy of $\Ctwo$ containing $I$.

Let $I_1,\ldots,I_c$ be the components of $I$ which contain at least one edge.
For each $j\in[c]$, choose a vertex $x_j\in V(I_j)$ (note that there are $v_j$ possible choices for this, which leads to a total of 
\begin{equation}\label{equa:bound1C4}
    \prod_{j=1}^cv_j\leq 2^{|I|}
\end{equation}
choices for $\{x_1,\ldots,x_c\}$).
Now, each ordering $\sigma$ of $[n]$ (recall this defines a copy of $\Ctwo$) induces an ordering on the set consisting of the vertices $x_1,\ldots,x_c$ as well as all isolated vertices in $I$.
Let us denote this induced ordering as $\tau=\tau(\sigma)$.
We now want to bound the total number of possible orderings $\sigma$ by first bounding the number of orderings $\tau$ (which depend on the choice of $x_1,\ldots,x_c$) and then the number of orderings $\sigma$ with $\tau=\tau(\sigma)$.

After the choice of $x_1,\ldots,x_c$, the number of possible orderings $\tau$ is \begin{equation}\label{equa:bound2C4}
    (n-|V(I)|+c)!.
\end{equation}
Now, in order to obtain some $\sigma$ such that $\tau=\tau(\sigma)$, it suffices to `insert' the vertices which are missing into the ordering, and this must be done in a way which is consistent with the structure of the components $I_j$.
For each $j\in[c]$, consider a labelling of the vertices of $I_j$ starting with $x_j$ and such that each subsequent vertex has at least one neighbour with a smaller label.
Then, we insert the vertices of $I_j$ into the ordering following this labelling, and note that, for each vertex, there are at most three choices, as $\Delta(I_j)\leq3$.
This implies there are at most $3^{|V(I_j)|-1}\leq3^{|I_j|}$ possible ways to fix the ordering of the vertices of $I_j$.
By considering all $j\in[c]$, we conclude that there are at most
\begin{equation}\label{equa:bound3C4}
    \prod_{j=1}^c 3^{|I_j|}\le 3^{|I|}
\end{equation} 
possible orderings $\sigma$ which result in the same $\tau$.

Combining \eqref{equa:bound1C4}, \eqref{equa:bound2C4} and \eqref{equa:bound3C4} with the fact that there are $2n$ distinct orderings $\sigma$ which result in the same copy of $\Ctwo$, we conclude that
\begin{equation}\label{equa:bound4C4}
|\cC\cap \langle I\rangle|\le \frac{6^{|I|}}{2n}(n-|V(I)|+c)!\le 6^{|I|}(n-|V(I)|+c-1)!.
\end{equation}

We can now estimate the spreadness of $\cC$. 
Consider first any $I\subseteq M$ with $|I|\leq n/10=|\Ctwo|/15$, and let $c$ be its number of components.
Then, by substituting the bound given by \cref{obs:small-edge-spread} into \eqref{equa:bound4C4}, we conclude that
\[|\cC\cap \langle I\rangle|\leq 6^{|I|}\left(n-\frac23|I|-\frac{c}{3}-1\right)!.\]
By using the bound on $|I|$ and taking into account that $|\cC|=(n-1)!/2$ and $|\Ctwo|=3n/2$, we conclude that $|\cC\cap \langle I\rangle|\leq q^{|I|}|\Ctwo|^{-c/3}|\cC|$ for $q\geq250n^{-2/3}$\COMMENT{Let $\ell\coloneqq |I|$.
It suffices to check that
\[6^{\ell}\left(n-\frac{2}{3}\ell-\frac{c}{3}-1\right)!\leq q^\ell\left(\frac32n\right)^{-c/3}\frac{(n-1)!}{2}.\]
By Stiring's approximation, it suffices to check that (for sufficiently large $n$)
\[q^\ell\geq4\cdot6^{\ell}\left(\frac{n}{n-\frac{2}{3}\ell-\frac{c}{3}}\right)^{1/2}e^{2\ell/3}\left(e\frac32\frac{n}{n-\frac{2}{3}\ell-\frac{c}{3}}\right)^{c/3}\left(\frac{n-\frac{2}{3}\ell-\frac{c}{3}}{n}\right)^n\left(n-\frac{2}{3}\ell-\frac{c}{3}\right)^{-2\ell/3}.\]
By taking roots, we want to have
\[q\geq\left(4\left(\frac{n}{n-\frac{2}{3}\ell-\frac{c}{3}}\right)^{1/2}\right)^{1/\ell}6e^{2/3}\left(e\frac32\frac{n}{n-\frac{2}{3}\ell-\frac{c}{3}}\right)^{c/(3\ell)}\left(\frac{n-\frac{2}{3}\ell-\frac{c}{3}}{n}\right)^{n/\ell}\left(n-\frac{2}{3}\ell-\frac{c}{3}\right)^{-2/3}.\]
Now consider each term in the expression above.
The term inside the first big parenthesis tends to $1$ as $\ell$ goes to infinity, and it is always bounded by some constant (say, it is at most $8$, since we have a upper bound on $\ell$ and, thus, also on $c$ that leads to the thing inside the second parenthesis being a constant smaller than $2$).
The term $6e^{2/3}$ remains as is.
The next term can be bounded by $(3e/2)^{1/3}$ since $c\leq\ell$.
The next term can be bounded simply by $1$.
The final term, then, is at most $(n/2)^{-2/3}$.
Putting all of these bound together, it follows that it suffices to have
\[q\geq8\cdot6e^{2/3}(3e/2)^{1/3}(n/2)^{-2/3}=2^{13/3}3^{4/3}e n^{-2/3}.\]
In particular, the constant above is at most $250$.}.

Similarly, assume $I\subseteq M$ has $|I|>n/10$ edges and $c$ components.
By substituting the bound given by \cref{obs:large-edge-spread} into \eqref{equa:bound4C4}, we now have that
\[|\cC\cap \langle I\rangle|\leq 6^{|I|}\left(n-\frac23|I|\right)!.\]
Now, as above, we conclude that $|\cC\cap \langle I\rangle|\leq q^{|I|}|\cC|$ for $q\geq12n^{-2/3}$.\COMMENT{Let $\ell\coloneqq|I|$.
It suffices to show that
\[6^\ell\left(n-\frac{2}{3}\ell\right)!\leq q^\ell\frac{(n-1)!}{2}.\]
By making use of Stirling's approximation, we have that $(n-2\ell/3)!\geq\sqrt{2\pi(n-2\ell/3)}((n-2\ell/3)/e)^{n-2\ell/3}$, and using the same approximation for $n!$, we conclude that it suffices to have 
\[q^\ell\geq4n\frac{\sqrt{n-2\ell/3}}{\sqrt{n}}\cdot6^\ell e^{2\ell/3}\frac{\left(n-\frac{2}{3}\ell\right)^{n-2\ell/3}}{n^n}.\]
Now note that 
\[4n\frac{\sqrt{n-2\ell/3}}{\sqrt{n}}\cdot6^\ell e^{2\ell/3}\frac{\left(n-\frac{2}{3}\ell\right)^{n-2\ell/3}}{n^n}\leq4n6^\ell e^{2\ell/3}\frac{\left(\left(1-\frac{2\ell}{3n}\right)n\right)^{n-2\ell/3}}{n^n}\leq4n6^\ell e^{2\ell/3}n^{-2\ell/3}.\]
Therefore, it suffices to have 
\[q\geq(4n)^{1/\ell}6 e^{2/3}n^{-2/3}.\]
But now, by the bound $\ell\geq n/10$, we have that $(4n)^{1/\ell}\to1$, so in particular the inequality holds (for sufficiently large $n$) if $q=6.1e^{2/3}n^{-2/3}\leq12n^{-2/3}$.}

Combining the two statements above, it follows by definition that $\cC$ is $((250n^{-2/3},1/3,1/15))$-superspread, as we wanted to see.
\end{proof}

\begin{proof}[Proof of \cref{thm:Ctwo-threshold}]
\Cref{obs:general-spread,obs:small-edge-spread} establish that $\cC$ is $(250n^{-2/3},1/3,1/15)$-superspread. 
By \cref{thm:main} we have that, if $p\ge Cn^{-2/3}$, where $C$ is a sufficiently large constant, then
\[
\PP\left[\Ctwo\subseteq \Gnp\right]\ge 1/2.
\]
To finish the argument, one can employ a general result of Friedgut~\cite{Fri05} (see, e.g., a recent paper of \citet{NS20}) which allows to establish that 
\[\PP\left[\Ctwo\subseteq G(n,(1+o(1))p)\right]=1-o(1).\qedhere\]
\end{proof}

\begin{comment}
\begin{observation}\label{obs:number-subgraphs}
Let $F$ be a subgraph of $\Ctwo$ with $f$ edges, then the number of subgraphs  $U$ of $F$ with $\ell$ edges and $c$ components is at most
\[
(12e)^{\ell}\binom{f}{c}.
\]
\end{observation}

\begin{proof}
The proof follows similar to the proof of Proposition~2.2 from~\cite{KNP20}: the number of connected $h$-edge subgraphs of $G$ containing a given vertex ist less than $(e\Delta(G))^h$.
Now we first specify the roots of the $c$ components of the $\ell$-edge subgraph of $F$ in at most $2^{c}\binom{f}{c}$ ways, then choose the sizes of the components in at most $\binom{\ell-1}{c-1}$ ways and then choose the subgraphs along $F$ in at most $\prod_{j=1}^c(3e)^{\ell_j}=(3e)^\ell$ ways.
In total we get at most $(3e)^\ell\binom{\ell-1}{c-1} 2^{c}\binom{f}{c}\le (12e)^{\ell}\binom{f}{c}$ subgraphs.
\end{proof}
\end{comment}

\subsection{Spanning \texorpdfstring{$K_r$}{Kr}-cycles}

In the following we will study copies of $K_r$ arranged in a cyclic way.
Since there are several ways how two consecutive copies of $K_r$ can overlap, we provide a precise definition of what will be called an $s$-overlapping $K_r$-cycle.

\begin{definition}\label{def:Krsn}
Let $r> s\ge 0$ and $n\in \NN$ with $(r-s)\mid n$ be integers.
A $\Krsn$-cycle is a graph on vertex set $\ZZ_n=[0,n-1]$ whose edge set is the union of the edge sets of $n/(r-s)$ copies of $K_r$, where for each $i\in[0,n/(r-s)-1]$ there is a copy of $K_r$ on the vertices $[i(r-s),i(r-s)+r-1]$ (modulo $n$).
\end{definition}

In other words, the $n/(r-s)$ copies of $K_r$ are arranged cyclically on the vertex set $\ZZ_n$, so that two consecutive copies of $K_r$ intersect in exactly $s$ vertices and two non-consecutive cliques intersect in as few vertices as possible.

The case $s=0$ corresponds to a $K_r$-factor.
The threshold for the property of containing a $K_r$-factor was famously determined by \citet{JohKahVu08}.
For the case $s=1$, the copies of $K_r$ in $K_{r,1,n}$ are edge-disjoint.
As mentioned in the introduction, the threshold for the appearance of $K_{r,1,n}$ in $\Gnp$ was recently determined by \citet{Frieze20}.
When $s=r-1$, the $s$-overlapping $K_r$-cycles are usually referred to as the $(r-1)$-th power of a Hamilton cycle $C_n$, where the $k$-th power of some arbitrary graph $G$ is obtained by connecting any two vertices of $G$ which are at distance at most $k$ with an edge.
The threshold for the appearance of $K_{r,r-1,n}$ is known to be $n^{-1/r}$.
This was observed by \citet{KO12} for $r\ge 4$, while the case $r=3$ was solved recently by \citet{KNP20}.

We determine the threshold for the appearance of $\Krsn$ for all the remaining values of $r$ and~$s$.
Whenever $s\geq3$, the result follows from a general result of \citet{Ri00}; see \cref{sec:conclude}.
Our main focus here is on the cases when $s=2$.
The overall strategy follows the same structure as in \cref{sect31}. 

We denote the set of all unlabelled copies of $\Krsn$ on $[n]$ by $\cCrsn$.
When talking about subgraphs of $\Krsn$, we refer to sets of consecutive vertices as \emph{segments}.
The \emph{length} of a segment is the number of vertices it contains.

First, we observe the following several simple facts about $\Krsn$.

\begin{fact}\label{fact:prop_Krsn}
Let $r> s\ge 0$ and $n\in \NN$ with $(r-s)\mid n$.
Then, the number of edges in $\Krsn$ is exactly 
\[\left(\binom{r}{2}-\binom{s}{2}\right)\frac{n}{r-s}=\frac{1}{2}(r+s-1)n.\] 
In particular, for $s\le r/2$, the number of vertices of degree $2r-s-1$ is exactly ${sn}/({r-s})$ (such vertices belong to two copies of $K_r$), whereas the remaining ${(r-2s)n}/{(r-s)}$ vertices have degree $r-1$ (these vertices belong to exactly one copy of $K_r$).\qed
\end{fact}
 
We will call vertices of $\Krtn$ with degree $2r-3$ \emph{heavy} and those with degree $r-1$ \emph{light}.

\begin{fact}\label{fact:number-cCrsn}
Let $r> s\ge 1$ and $n\in \NN$ with $(r-s)\mid n$ and $s\le r/2$. 
We have 
\[|\cCrsn|=\frac{(n-1)! (r-s)}{2((r-2s)!)^{n/(r-s)}(s!)^{n/(r-s)}}=\frac{r-s}{2}d_{r,s}^n(n-1)!,\]
where $d_{r,s}\in (0,1]$ is some absolute constant that depends on $s$ and $r$ only.\COMMENT{There are $n!$ permutations of the vertices.
Given this, there are $2n/(r-s)$ equal cyclic orders (the starting point matters up to $r-s$ choices, and after that the order would repeat itself, and we can go in either of the two directions).
Furthermore, for each set of consecutive vertices of degree $r-1$, we can reorder them in any way and obtain the same graph; there are $n/(r-s)$ sets of $r-2s$ such consecutive vertices.
Similarly, for each set of consecutive vertices of degree $2r-s-1$ contained in the same edges, we can reorder them in any way and obtain the same graph; there are $n/(r-s)$ sets of $s$ such consecutive vertices.} \qed
\end{fact}

\begin{fact}\label{fact:consec_Krsn}
Let $r\ge 4$ and $n\in \NN$ with $(r-2)\mid n$.
Let $V\subseteq V(\Krtn)$ be a segment starting in the first vertex of some clique $K_r$ with $|V|\le n/(2r)+1$\COMMENT{This bound (or a weaker form) is necessary in the sense that, if the graph wraps around, then the number of edges can be slightly larger than described below.}.
Then, 
\begin{equation}\label{eq:max-subgraph}
e(\Krtn[V])=\left(\binom{r}{2}-1\right)a+\binom{b}{2}-\max\{2-b,0\}\cdot(r-2),
\end{equation}
where $|V|=(r-2)a+b$ with $a,b\in \NN_0$ and $0\le b< r-2$.\COMMENT{
\begin{proof}
Let $v=(r-2)a+b$ with $a,b\in \NN_0$ and $0\le b< r-2$. 
We may assume that $V=[0,v-1]$. 
We distinguish the following cases: $b\in\{0, 1\}$ and $2\le b< r-2$.
Assume first that $b\in\{0, 1\}$.
Then, we have 
\[
e(I)=\left(\binom{r}{2}-1\right)a+\binom{b}{2}-(2-b)(r-2)=\left(\binom{r}{2}-1\right)(a-1)+\binom{r-2+b}{2}
\]
(we have $a$ contributions of $\binom{r}{2}-1$, where we count edges ($a$ times) on a set of size $r-2$ and between this set and the next two vertices of $V$ (note this guarantees that we do not count the edge induced by said pair of vertices twice); additionally we get $\binom{b}{2}$ edges, but we have to correct this for the last `full' clique in $V$, since it will miss exactly $2-b$ vertices and, therefore, $(2-b)(r-2)$ edges).\\
Assume next that $2\le b< r-2$.
By the same argument as above, we have $e(I)=\left(\binom{r}{2}-1\right)a+\binom{b}{2}$.
\end{proof}
} \qed
\end{fact}

Instead of using \eqref{eq:max-subgraph}, we will make use of the following estimate to streamline our calculations.

\begin{proposition}\label{prop:easy-bound}
Let $r\ge 4$ and $v\in \NN$ with $v=(r-2)a+b$, where $a,b\in \NN_0$ and $0\le b< r-2$. 
Then,
\begin{equation}\label{eq:easy-bound}
\left(\binom{r}{2}-1\right)a+\binom{b}{2}-\max\{2-b,0\}\cdot(r-2)\le \frac{r+1}{2}v -\frac{r+2}{2}.
\end{equation}
\end{proposition}

\begin{proof}
We can rewrite the LHS of~\eqref{eq:easy-bound} as
\[
\frac{(r+1)(r-2)}{2}a+\binom{b}{2}-\max\{2-b,0\}\cdot(r-2).
\]
By substituting $v=(r-2)a+b$ in the RHS, we see that~\eqref{eq:easy-bound} is equivalent to\COMMENT{We want to prove
\[\frac{(r+1)(r-2)}{2}a+\binom{b}{2}-\max\{2-b,0\}\cdot(r-2)\le\frac{r+1}{2}((r-2)a+b) -\frac{r+2}{2}=\frac{(r+1)(r-2)}{2}a+\frac{r+1}{2}b-\frac{r+2}{2},\]
and the leftmost term cancels out.}
\[\binom{b}{2}-\max\{2-b,0\}\cdot(r-2)\le\frac{r+1}{2}b-\frac{r+2}{2}.\]
To verify this, we consider 
\begin{align*}
    f(b)\coloneqq&\, 2\left(\frac{r+1}{2}b-\frac{r+2}{2}-\left(\binom{b}{2}-\max\{2-b,0\}\cdot(r-2)\right)\right)\\
    =&\,(r+2-b)b-(r+2)+2\max\{2-b,0\}\cdot(r-2).
\end{align*}
For all $b\neq2$ we have $f'(b)=(r+2)-2b-2(r-2)\cdot \mathds{1}_{\{b<2\}}$ and $f''(b)=-2$. 
Since $f$ is concave in $(-\infty,2)$ and $(2,\infty)$, in order to verify that $f(b)\ge 0$ for all $b\in[0,r-3]$ (which is then equivalent to~\eqref{eq:easy-bound}) it suffices to check the value of $f(b)$ at $0$, $1$, $2$ and $r-3$ (assuming this is larger than $2$)\COMMENT{$1$ is necessary for the case $r=4$.}.
Indeed,
\begin{align*}
    f(0)&=-(r+2)+4(r-2)=3r-10> 0,\\
    f(1)&=(r+1)-(r+2)+2(r-2)=2r-5>0,\\
    f(2)&=2r-(r+2)=r-2>  0,\\
\intertext{and, if $r-3> 2$,}
    f(r-3)&=5(r-3)-(r+2)=4r-17>0.\qedhere
\end{align*}
\end{proof}

Our goal now is to establish that the densest subgraphs of $\Krtn$ are precisely those described in \cref{fact:consec_Krsn}.
The next lemma establishes that, among all subgraphs of $\Krtn$ induced by segments of length at most $n/(2r)+1$ (i.e., there is no `wrapping around the cycle'), the densest ones are those where the segment starts in a `new' $K_r$ or ends in a `full' $K_r$.

\begin{lemma}\label{lem:densest-case-segment}
Let $r\ge 4$ and $n\in \NN$ with $(r-2)\mid n$.
Let $V\subseteq V(\Krtn)$ be a segment with $r\le |V|\le n/(2r)+1$.
Then, the number of edges induced by $V$ is 
maximised when $V$ starts in the first vertex of some clique $K_r$ or ends in the last vertex of some clique $K_r$.
\end{lemma}

\begin{proof}
Let $V$ be a segment which induces the maximum possible number of edges from $\Krtn$.
By the symmetries of $\Krtn$, we may assume that $V\cap([0,r-1]\cup[n-r,n-1])=\varnothing$\COMMENT{This assumption is in place so there are no issues when talking about last and first cliques.}.
Assume that $V$ is not of the form described in the claim (i.e., it neither begins in the first vertex nor ends in the last vertex of some clique $K_r$).
Let $i_1$ and $i_2$ be the first and last vertices of $V$, and let $j_1$ and $j_2$ be the number of vertices which $V$ contains in the (last) clique $K_r$ which contains $i_1$ and in the (first) clique which contains $i_2$, respectively.
We may assume, without loss of generality, that $j_1\le j_2$ (and recall that $j_1,j_2<r$).
However, then the set $(V\setminus\{i_1\})\cup\{i_2+1\}=[i_1+1,i_2+1]$ induces more edges than $V$\COMMENT{We remove $j_1-1$ edges from one side and add $j_2$ edges to the other.}.
But this contradicts our choice of $V$ as a set of consecutive vertices which induces the maximum possible number of edges.
\end{proof}

Now we prove that no subgraph of $\Krtn$ is denser than the subgraphs induced by segments.

\begin{lemma}\label{lem:densest-case-general}
Let $r\ge 4$ and $n\in \NN$ with $(r-2)\mid n$.
Let $V\subseteq V(\Krtn)$ with $r\le v\coloneqq|V|\le n/(2r)+1$.
Then, the number of edges induced by $V$ is at most 
the number induced by a segment of length $v$.
\end{lemma}

\begin{proof} 
Let $V\subseteq V(\Krtn)$ be a set of cardinality $v$ inducing the maximum possible number of edges.
Let $I$ be the graph induced by $V$.
Of course, we may assume $I$ has no isolated vertices.
Let $S$ be a smallest segment containing $V$.
By the symmetries of $\Krtn$, we may assume that $S\cap([0,r-1]\cup[n-r,n-1])=\varnothing$.
We use a compression-type argument to show that we can modify $V$ into a segment which induces at least as many edges as $V$.
We achieve this by consecutively creating new sets $V'$ which are contained in shorter segments but induce at least as many edges as the previous set.

Let $i_1$ and $i_2$ be the first and last vertices of $V$ (i.e., $S=[i_1,i_2]$) and notice that $i_1$ and $i_2$ do not form an edge (since $v\le n/(2r)+1$).
Observe, then, that $\deg_I(i_1),\deg_I(i_2)\in[r-1]$.
By an argument as in \cref{lem:densest-case-segment}, we may assume, without loss of generality, that $\deg_I(i_1)\ge \deg_I(i_2)$ and that $i_1$ is the first vertex of some clique $K_r$ completely contained in $I$:
otherwise, we could replace the vertex $i_2$ with some missing vertex from such a clique and increase the number of edges\COMMENT{Assuming that $\deg_I(i_1)\ge \deg_I(i_2)$ can be done by symmetry.
Now, assume $i_1$ is not the first vertex of a `full' copy of $K_r$.
By deleting $i_2$, we loose $\deg_I(i_2)$ edges. By then adding a new vertex in the clique containing $i_1$, since this clique must already contain $\deg_I(i_1)+1$ vertices, we gain $\deg_I(i_1)+1>\deg_I(i_2)$ new edges, so the total number of edges goes up.
But this contradicts the assumption that $I$ has the maximum possible number of edges.}.

Let $K^{(1)}$ be the copy of  $K_r$ contained in $I$ with the smallest indices (in particular, it contains $i_1$). 
Assume that $V$ is not a segment.
Consider the vertex $i'\in S\setminus V$ with the smallest index.
Observe that $i'\notin V(K^{(1)})$.
We distinguish two cases, depending on whether $i'$ is heavy or light.

If $i'$ is heavy, let $K'$ and $K''$ be the two copies of $K_r$ from $\Krtn$ with $i'\in V(K')\cap V(K'')$ and $K'$ containing smaller indices than $K''$.
If $E(K'')\cap E(I)=\varnothing$, then we can shift all edges of $I$ induced by $V\cap[i'+1,i_2]$ to the left $r$ positions, yielding a graph contained in a segment of length $i_2-i_1+1-r$ with the same number of edges as $I$.
Hence, we assume that $E(K'')\cap E(I)\neq\varnothing$, which implies $V$ contains at least two of the vertices of $K''$. 
Observe that $K^{(1)}\neq K'$.
Then, replace $i_1$ by $i'$.
In this way, since $i_1$ is the first vertex from $K^{(1)}$, we remove $r-1$ edges, but at the same time we add at least $r$ edges\COMMENT{All vertices in $K'$ before $i'$ must be in the graph, so at least $r-2$ in $V(K')\setminus V(K'')$. But also, since $|V\cap V(K'')|\geq2$, we are adding at least two more edges.\\
We could ignore the fact that $i'\notin K^{(1)}$ and still obtain the desired result (we would replace $e-1$ edges by $r-1$ new edges and obtain a shorter segment).}.
But this contradicts our choice of $V$.

Assume now that $i'$ is light and let $K$ be the unique clique $K_r$ from $\Krsn$ with $i'\in V(K)$.
Since $i'$ is light, by its definition we must have $|V(K)\setminus V|\le r-2$\COMMENT{There have to be at least two heavy vertices in $K$ which lie in $V$ (the last two with indices below $i'$).}.
We replace the first $t\coloneqq |V(K)\setminus V|\le r-2$ vertices from $K^{(1)}$ by adding $V(K)\setminus V$ to $V$.
In this way, we remove $\sum_{i=1}^{t} (r-i)$ edges, but at the same time we add at least $\sum_{i=1}^{t} (r-i)$ edges.
Since $i_2>i'$, this means the new set lies in a shorter segment but induces at least as many edges.

In all described situations we managed to make the segment containing $V$ shorter while not decreasing the number of edges.
Hence, we eventually find a segment $V'$ inducing at least as many edges as $V$.
\end{proof}

The following now follows directly by combining \cref{fact:consec_Krsn,prop:easy-bound,lem:densest-case-segment,lem:densest-case-general} (and noting that the bound holds trivially if $|V|\leq r$).

\begin{corollary}\label{coro:Krdensity_bound}
Let $r\ge 4$ and $n\in \NN$ with $(r-2)\mid n$.
Let $V\subseteq V(\Krtn)$ with $|V|\le n/(2r)+1$.
Then,
\[e(\Krtn[V])\leq \frac{r+1}{2}|V|-\frac{r+2}{2}.\]
\end{corollary}

Using what we have proved so far, we can obtain estimates which will be crucial for studying the spreadness of $\cCrtn$.

\begin{lemma}\label{lemma:small-spread-new}
Let $r\ge 4$ and $n\in \NN$ with $(r-2)\mid n$.
Let $I\subseteq\Krtn$ be a subgraph with $\ell\le n/(2r)$ edges and $c$ components.
Then,
\[|V(I)|-c\ge \frac{2}{r+1}\ell+\frac{c}{r+1}.\]
\end{lemma}

\begin{proof}
Let $I_1,\ldots,I_c$ be the components of $I$ with at least one edge, and let $v_1,\ldots,v_c$ be the number of vertices spanned by $I_1,\ldots,I_c$, respectively. 
For each $j\in[c]$, since $|I_j|\le|I|\le n/(2r)$, by \cref{coro:Krdensity_bound} we have the following easy bound:
\[|I_j|\le \frac{r+1}{2}v_j-\frac{r+2}{2}.\]
Summing over all $j\in[c]$, we obtain that 
\[\ell=|I|\le \frac{r+1}{2}|V(I)| -\frac{r+2}{2}c=\frac{r+1}{2}\left(|V(I)|-c\right)-\frac{c}{2},\]
and the claim follows by reordering.
\end{proof}

\begin{lemma}\label{lemma:large-spread-new}
Let $r\ge 4$ and $n\in \NN$ with $(r-2)\mid n$.
Let $I\subseteq\Krtn$ be a subgraph with $\ell$ edges and $c$ components.
Then,
\[|V(I)|\ge \frac{2}{r+1}\ell.\]
\end{lemma}

\begin{proof}
The vertex set of $\Krtn$ consists of $t\coloneqq n/(r-2)$ segments of heavy vertices and $t$ segments of light vertices, which alternate as we traverse the vertex set.
The segments of heavy vertices have length $2$, and the segments of light vertices have length $r-4$ (the case $r=4$ is special: here the segments of light vertices are empty).
For each $i\in[t]$, let $h_i$ and $\ell_i$ denote the number of heavy vertices and light vertices of $I$ in the $i$-th segment of heavy or light vertices, respectively.
For notational purposes, let $h_{t+1}\coloneqq h_1$.
Then, we can bound the number of edges of $I$ as follows:
\begin{align*}
|I|&\le \sum_{i=1}^t \left(\binom{h_i}{2}+\binom{\ell_i}{2}+h_i\ell_i+h_{i+1}\ell_i+h_{i+1}h_i\right)\\
&\le \sum_{i=1}^t \left(\binom{h_i}{2}+\binom{\ell_i}{2}+h_i\ell_i+2\ell_i+2h_i\right)=\sum_{i=1}^t \left(\binom{h_i+\ell_i}{2}+2(\ell_i+h_i)\right).
\end{align*}
Next, observe that, for each $i\in[t]$,
\[\binom{h_i+\ell_i}{2}+2(\ell_i+h_i)=(h_i+\ell_i)\left(\frac{h_i+\ell_i-1}{2}+2\right)\le \frac{r+1}{2}(h_i+\ell_i),\]
where the inequality holds since $h_i+\ell_i\le r-2$.
The conclusion follows by adding over all $i\in[t]$.
\end{proof}

Combining the previous two lemmas, we show that $\cC_{r,2,n}$ is a $(O(n^{-2/(r+1)}),1/(r+1),1/(r(r+1)))$-superspread hypergraph.

\begin{lemma}\label{lem:spread}
Let $r\ge 4$ and $n\in \NN$ with $(r-2)\mid n$.
Then, the hypergraph $\cC_{r,2,n}$ of all copies of $\Krtn$ in $M=\binom{[n]}{2}$ is  $(O(n^{-2/(r+1)}),1/(r+1),1/(r(r+1)))$-superspread.
\end{lemma}

\begin{proof}
Let $I\subseteq M$.
Our first aim is to obtain a general upper bound on $|\cCrtn\cap \langle I\rangle|$.
If $I$ is not contained in any copy of $\Krtn$, we automatically have $|\cCrtn\cap \langle I\rangle|=0$, so we may assume $I$ is a subgraph of some copy of $\Krtn$.
Recall that each copy of $\Krtn$ can be defined by an ordering of the $n$ vertices, so it suffices to bound the number of orderings which yield a copy of $\Krtn$ containing $I$.

Let $I_1,\ldots,I_c$ be the components of $I$ with at least one edge.
For each $j\in[c]$, let $x_j\in V(I_j)$ (there are $v_j$ such possible choices, which leads to a total of
\begin{equation}\label{equa:Krtnboundspread1}
    \prod_{j=1}^cv_j\leq 2^{|I|}
\end{equation}
choices for $\{x_1,\ldots,x_j\}$\COMMENT{We have $v_j\leq|I_j|+1\leq 2^{|I_j|}$.}).
Then, each ordering $\sigma$ of $[n]$ which defines a copy of $\Krtn$ containing $I$ induces a unique ordering $\tau=\tau(\sigma)$ on the set consisting of $x_1,\ldots,x_j$ and all other isolated vertices.
The total number of such orderings $\tau$ is 
\begin{equation}\label{equa:Krtnboundspread2}
    (n-|V(I)|+c)!
\end{equation}
so now it suffices to bound, for each such $\tau$, the number of orderings $\sigma$ with $\tau=\tau(\sigma)$.

Given an ordering $\tau$, in order to obtain an ordering $\sigma$ with $\tau=\tau(\sigma)$, it suffices to `insert' the missing vertices into the ordering.
That is, for each $j\in[c]$, we need to `insert' the other vertices of $V(I_j)$ into the ordering.
By considering a labelling of the vertices of $I_j$ in such a way that each subsequent vertex is a neighbour of at least one previously included vertex (and taking into account that $x_j$ is already included), we note that there are at most $2r$ choices for each vertex (recall that $\Delta(\Krtn)<2r$).
This leads to a total of at most $(2r)^{|V(I_j)|-1}\leq(2r)^{|I_j|}$ possible ways to include the component $I_j$.
By considering all $j\in[c]$, we conclude that there are at most 
\begin{equation}\label{equa:Krtnboundspread3}
    \prod_{j=1}^c(2r)^{|I_j|}=(2r)^{|I|}
\end{equation}
orderings $\sigma$ with $\tau=\tau(\sigma)$.

Combining \eqref{equa:Krtnboundspread1}, \eqref{equa:Krtnboundspread2} and \eqref{equa:Krtnboundspread3} with the fact that each copy of $\Krtn$ is given by $d_{r,2}^{-n}2n/(r-s)$ distinct orderings (see \cref{fact:number-cCrsn}), we conclude that 
\begin{equation}\label{equa:spreadboundKr2n}
    |\cCrtn\cap \langle I\rangle|\le \frac{r-s}{2}d_{r,2}^n(4r)^{|I|}(n-|V(I)|+c-1)!.
\end{equation}

We can now estimate the spreadness of $\cCrtn$.
Consider first any $I\subseteq M$ with $|I|>n/(2r)$.
Note that $I$ has at most $2r$ components $I_j$ of size larger than $n/(2r)$.
For each of these components, we use \cref{lemma:large-spread-new} to bound $|V(I_j)|$.
For the remaining components $I_j$, we simply use the bound $|V(I_j)|-1\geq2|I_j|/(r+1)$, which follows by \cref{lemma:small-spread-new}.
By substituting these bounds into \eqref{equa:spreadboundKr2n}, we conclude that
\[|\cCrtn\cap\langle I\rangle|\leq\frac{r-s}{2}d_{r,2}^n(4r)^{|I|}\left(n-\frac{2}{r+1}|I|+2r\right)!.\]
By comparing this with the expression given in \cref{fact:number-cCrsn} (and taking into account the bound on $|I|$), we conclude that $|\cCrtn\cap\langle I\rangle|\leq q^{|I|}|\cCrtn|$ whenever $q\geq c_1 n^{-2/(r+1)}$\COMMENT{By \cref{fact:number-cCrsn}, we have that
\[|\cCrtn|=\frac{r-2}{2}d_{r,2}^n(n-1)!.\]
Now it suffices to satisfy
\[\frac{r-2}{2}d_{r,2}^n(4r)^{|I|}\left(n-\frac{2}{r+1}|I|+2r\right)!\leq q^{|I|}\frac{r-2}{2}d_{r,2}^n(n-1)!.\]
By Stirling's approximation, for sufficiently large $n$, it suffices to satisfy
\[2(4r)^{|I|}\left(n-\frac{2}{r+1}|I|+2r\right)_{2r}\sqrt{2\pi\left(n-\frac{2}{r+1}|I|\right)}\left(\frac{n-\frac{2}{r+1}|I|}{e}\right)^{n-\frac{2}{r+1}|I|}\leq q^{|I|}\frac{1}{n}\sqrt{2\pi n}\left(\frac{n}{e}\right)^n.\]
Rearranging,
\[q\geq\left(2\sqrt{n\left(n-\frac{2}{r+1}|I|\right)}\left(n-\frac{2}{r+1}|I|+2r\right)_{2r}\right)^{1/|I|}4re^{\frac{2}{r+1}}\left(\frac{n-\frac{2}{r+1}|I|}{n}\right)^{n/|I|}\left(n-\frac{2}{r+1}|I|\right)^{-2/(r+1)}.\]
Now, using the bound $|I|\geq n/(2r)$, as $n$ goes to infinity, the first term tends to $1$, so we have that, if $n$ is sufficiently large, it suffices to have
\[q\geq8re^{\frac{2}{r+1}}\left(1-\frac{1}{r(r+1)}\right)^{2r}\left(1-\frac{1}{r(r+1)}\right)^{-2/(r+1)}n^{-2/(r+1)}=c_1n^{-2/(r+1)}.\]}, where $c_1$ is a constant that depends only on $r$.

Consider now some $I\subseteq M$ with $|I|\leq n/(2r)=|\Krtn|/(r(r+1))$ (see \cref{fact:prop_Krsn}), and let $c$ be its number of components.
By making use of \cref{lemma:small-spread-new} and \eqref{equa:spreadboundKr2n}, we have that 
\[|\cCrtn\cap\langle I\rangle|\leq\frac{r-s}{2}d_{r,2}^n(4r)^{|I|}\left(n-\frac{2}{r+1}|I|-\frac{c}{r+1}-1\right)!.\]
As above, by comparing this with the expression given in \cref{fact:number-cCrsn}, and taking into account also \cref{fact:prop_Krsn}, for $q\geq c_2 n^{-2/(r+1)}$, where $c_2$ depends only on $r$, we have that $|\cCrtn\cap\langle I\rangle|\leq q^{|I|}|\Krtn|^{-c/(r+1)}|\cCrtn|$.
\end{proof}

\begin{proof}[Proof of \cref{thm:Kr-cycle}]
By \cref{lem:spread}, we have that $\cCrtn$ is $(O(n^{-2/(r+1)}),1/(r+1),1/(r(r+1)))$-superspread. 
By \cref{thm:main} it follows that, if $C$ is sufficiently large, for $p\ge Cn^{-2/(r+1)}$ we have
\[
\PP\left[\Krtn\subseteq \Gnp\right]\ge 1/2.
\]
To finish the argument, one can employ a general result of Friedgut~\cite{Fri05} (see also~\cite{NS20}) which allows to establish that 
$\PP\left[\Krtn\subseteq G(n,(1+o(1))p)\right]=1-o(1)$.
\end{proof}

\section{Concluding remarks}\label{sec:conclude}

\subsection{Dense overlapping \texorpdfstring{$K_r$}{Kr}-cycles}

As mentioned in the introduction, a general result of \citet{Ri00} provides a sufficient condition for a spanning graph to be contained in $\Gnp$. 
For a graph $H=(V,E)$, let $v(H)\coloneqq|V|$ and $e(H)\coloneqq|E|$.
For each integer $v$, let $e_H(v)\coloneqq\max \{ e(F) : F \subseteq H, v(F)=v \}$.
Then, the following parameter will be responsible for the upper bound on the threshold for the property that $H\subseteq \Gnp$:
\begin{align*}
\gamma(H) \coloneqq \max_{3 \le v \le n} \left\{ \frac{e_H(v)}{v-2} \right\}.
\end{align*}
Riordan proved the following (see also~\cite{PP15} for its generalization to hypergraphs).

\begin{theorem}\label{thm:Riordan}
Let  $H=H^{(i)}$ be a  sequence of graphs with $n=n(i)$ vertices (where $n$ tends to infinity with $i$), $e(H)=\alpha \binom{n}{2} = \alpha (n)\binom{n}{2}$ edges and $\Delta=\Delta(H)$.
Let $p = p(n)\colon \NN\to [0,1)$. 
If $H$ has a vertex of degree at least $2$ and $n p^{\gamma(H)} \Delta^{-4} \rightarrow \infty$, then a.a.s.~the random graph $\Gnp$ contains a copy of $H$.
\end{theorem}

From \cref{fact:prop_Krsn} it follows that $\gamma(\Krsn)\ge \frac{r+s-1}{2}$ and, since $\gamma(K_r)>\frac{r+s-1}{2}$ for $s\in[2]$, Riordan's theorem does not provide optimal bounds on the theshold in the case of our \cref{thm:Kr-cycle} (nor for $K_{r,1,n}$, for which the threshold was determined by \citet{Frieze20}).
However, in the cases for $s\ge 3$, Riordan's theorem suffices and yields the correct threshold $n^{-2/(r+s-1)}$ for the property that $\Gnp$ contains a copy of $\Krsn$. 

\subsection{Extensions: hypergraphs and rainbow thresholds}

Throughout this paper, for simplicity, we have focused on properties of random graphs.
However, we believe that \cref{thm:main} extends to random hypergraphs without much issue.

Very recently, Frieze and Marbach~\cite{FM21} extended the results from~\cite{FKNP19,KNP20} to rainbow versions, where the vertices of some $r$-uniform hypergraph are colored randomly with $r$ colors.
It is then shown in~\cite{FM21} that the upper bounds on the thresholds as proved in~\cite{FKNP19,KNP20} remain asymptotically the same to yield a rainbow hyperedge or a rainbow copy of a spanning structure (e.g., bounded degree spanning tree, square of a Hamilton cycle), and the result is also extended to a rainbow version of the containment of the $k$-th power of a Hamilton cycle.
We believe that the fragmentation lemma, \cref{lem:fragmentation}, and \cref{thm:main} also admit rainbow versions.

\bibliographystyle{mystyle}
\bibliography{thresholds}

\appendix

\section{Proof of Lemma~\ref{lem:fragmentation}} \label{app:Fragmentation}

We begin with the proof of the auxiliary \cref{lem:num_subgraphs}.

\begin{proof}[Proof of \cref{lem:num_subgraphs}]
The proof follows along the same lines as the proof of Proposition~2.2 from~\cite{KNP20}: one uses the fact that the number of connected $h$-edge subgraphs of a graph $G$ containing a given vertex is less than $(e\Delta(G))^h$.

We now bound the number of ways in which we may construct subgraphs with $\ell$ edges and $c$ components. 
First, we specify the roots of the $c$ components of the $\ell$-edge subgraph of $F$, which can be done in at most $2^{c}\binom{f}{c}$ ways\COMMENT{For each root, we first choose an edge, and then one vertex incident to this edge (note we do not have any bound on the number of vertices by assumption).}.
Next, we choose the sizes of the components; this can be done in at most $\binom{\ell-1}{c-1}$ ways (this is the number of $c$-compositions of $\ell$).
Say that, to each of the components $j\in[c]$, we have assigned size $\ell_j$.
We finally choose the subgraphs which conform each component along $F$, which, by the fact mentioned in the first paragraph, can be done in at most $\prod_{j=1}^c(ed)^{\ell_j}=(ed)^\ell$ ways.
This yields a total of at most $(ed)^\ell\binom{\ell-1}{c-1} 2^{c}\binom{f}{c}\le (4ed)^{\ell}\binom{f}{c}$ subgraphs.
\end{proof}

The proof of \cref{lem:fragmentation} is essentially from~\cite[Lemma~3.1]{KNP20}: 
the only changes we need to make are replacing $2n$ in~\cite{KNP20} by $k_i$ and slightly adapting the computations for the verification of \cref{claim:pathological} in the pathological case below.

Since we will follow the proof of Lemma~3.1 from~\cite{KNP20} closely, we use the same notation wherever possible, so that the reader familiar with the argument can quickly verify the validity of \cref{lem:fragmentation}. 

\begin{proof}[Proof of \cref{lem:fragmentation}]
Without loss of generality, we may assume that $n$ is sufficiently large, and thus $k_i$ and $k$ are also sufficiently large.
We may also assume that $\cH_i$ is $k_i$-uniform (indeed, every set $S\in\cH_i$ is contained in some $S'\in\cF$, so we can add some arbitrary $k_i-|S|$ vertices from $S'\setminus S$ to $S$).
Let $M\coloneqq \binom{[n]}{2}$ and $m\coloneqq |M|$.

Now, consider pairs $(S,X)$ with $S\in\mathcal{H}_i$ and $X\subseteq M$ such that $|S\cap X|=t$, for some $t\in[0,k_i]$\COMMENT{Note, actually, that we may assume $t>k$, as otherwise the pairs are good by definition ($S$ itself will yield a good fragment).}.
Our aim is to prove that, for each $t$,
\begin{equation}\label{eq:intersection_t}
\left|\left\{(S,X):|S\cap X|=t,(S,X)\text{ is $k$-bad}\right\}\right|\leq 2C^{-k/3}|\cH_i|\binom{k_i}{t}\binom{m-k_i}{w-t}.
\end{equation}
Summing over all $t\in[0,k_i]$ yields the counting version of the bound~\eqref{eq:fragmentation}.

Let us bound the number of bad pairs $(S,X)$ with $|S\cap X|=t$, for some fixed $t\in[0,k_i]$.
Set $w'\coloneqq w-t$, so $|X\setminus S|=w'$ and $|X\cup S|=w'+k_i$.
A set $Z\in\binom{M}{w'+k_i}$ will be called \emph{pathological} if 
\[
\left|\left\{S\in \cH_i : S\subseteq Z,(S,Z\setminus S) \text{ is $k$-bad} \right\}\right|>C^{-k/3}|\cH_i|\binom{m-k_i}{w'}\bigg/\binom{m}{w'+k_i}.
\]
A pair $(S,X)$ with $|S\cup X|=w'+k_i$ will be called \emph{pathological} if $S\cup X$ is pathological.
We estimate the contributions of pathological and nonpathological pairs to \eqref{eq:intersection_t} separately.

\textbf{Nonpathological pairs.}
To specify a nonpathological pair $(S,X)$ with $|S\cap X|=t$, one first specifies $Z=S\cup X$, then $S$, and finally $X$. 
We use the crude bound $\binom{m}{w'+k_i}$ for the number of sets~$Z$.
Since $Z$ is nonpathological, there are at most $C^{-k/3}|\cH_i|\binom{m-k_i}{w'}\big/\binom{m}{w'+k_i}$ choices for $S$, since $(S,Z\setminus S)$ is $k$-bad if $(S,X)$ is $k$-bad.
Finally, one can specify $X\subseteq Z$ in at most $\binom{k_i}{t}$ ways, since $|X\cap S|=t$ and $|S|=k_i$.
Altogether, this gives a total of at most 
\begin{equation}\label{equa:nonpatho}
    \binom{m}{w'+k_i}\cdot \binom{k_i}{t}\cdot C^{-k/3}|\cH_i|\binom{m-k_i}{w'}\bigg/\binom{m}{w'+k_i}\overset{w'=w-t}{=}C^{-k/3}|\cH_i|\binom{m-k_i}{w-t} \binom{k_i}{t}
\end{equation}
$k$-bad pairs.

\textbf{Pathological pairs.} 
The following claim will assist in estimating pathological contributions.

\begin{claim}\label{claim:pathological}
For a given $S\in\cH_i$ and\/ $Y$ chosen uniformly at random from $\binom{M\setminus S}{w'}$, we have that
\begin{equation}\label{eq:claim}
\EE\left[\left|\left\{J\in\cH_i: J\subseteq S\cup Y,|J\cap S|\ge k\right\}\right|\right]\le C^{-2k/3}|\cH_i|\binom{w'+k_i}{k_i}\bigg/\binom{m}{k_i}.
\end{equation}
\end{claim}

The number of pathological pairs $(S,X)$ can now be estimated as follows.
First, we choose~$S$ and $S\cap X$, which can be done in at most $|\cH_i|\binom{k_i}{t}$ ways.
We then need to choose $X\setminus S$. 
Since we are counting pairs $(S,X)$ which are $k$-bad, for every $J\in\cH_i$ with $J\subseteq S\cup X$ we have that $|J\cap S|\ge |J\setminus X|\ge k$\COMMENT{Strict inequality, actually.}.
Since we are only considering pathological pairs $(S,X)$, by definition this yields the following lower bound on the number of such $J\in\cH_i$\COMMENT{By definition, we have that 
\[\left|\left\{J\in \cH_i : J\subseteq S\cup X,(J,(S\cup X)\setminus J) \text{ is $k$-bad} \right\}\right|>C^{-k/3}|\cH_i|\binom{m-k_i}{w'}\bigg/\binom{m}{w'+k_i}.\]
Now it suffices to check that 
\[\left\{J\in \cH_i : J\subseteq S\cup X,(J,(S\cup X)\setminus J) \text{ is $k$-bad} \right\}\subseteq \left\{J\in\mathcal{H}_i:J\subseteq S\cup (X\setminus S), |J\cap S|\ge k\right\}.\]
Indeed, the conditions that $J\subseteq S\cup X$ or $J\subseteq S\cup (X\setminus S)$ are equivalent, so consider an arbitrary such $J$ such that $(J,(S\cup X)\setminus J)$ is $k$-bad.
Now, since $S\subseteq J\cup((S\cup X)\setminus J)$, the definition of $k$-bad means that, in particular, $|S\setminus((S\cup X)\setminus J)|>k$, and note that $S\setminus((S\cup X)\setminus J)=S\cap J$.
So indeed the containment above holds.}:
\begin{align*}
\left|\left\{J\in\mathcal{H}_i:J\subseteq S\cup (X\setminus S), |J\cap S|\ge k\right\}\right|&\ge C^{-k/3} |\cH_i|\binom{m-k_i}{w'}\bigg/\binom{m}{w'+k_i}\\
&=C^{-k/3} |\cH_i|\binom{w'+k_i}{k_i}\bigg/\binom{m}{k_i}.
\end{align*}
When choosing a uniformly random set $X\setminus S$ of size $w'$, \cref{claim:pathological} and Markov's inequality yield 
\[
\PP\left[\left|\left\{J\in\mathcal{H}_i:J\subseteq S\cup (X\setminus S), |J\cap S|\ge k\right\}\right|\ge C^{-k/3} |\cH_i|\binom{m-k_i}{w'}\bigg/\binom{m}{w'+k_i}\right]\le  C^{-k/3},
\]
which gives us at most $C^{-k/3}\binom{m-k_i}{w'}=C^{-k/3}\binom{m-k_i}{w-t}$ choices for $X\setminus S$. Altogether this yields at most 
\begin{equation}\label{equa:patho}
    C^{-k/3}\binom{m-k_i}{w-t}|\cH_i|\binom{k_i}{t}
\end{equation}
choices for pathological fragments.

From the bounds in each case, \eqref{equa:nonpatho} and \eqref{equa:patho}, we obtain~\eqref{eq:intersection_t}.
\end{proof}

Finally, we turn to the proof of \cref{claim:pathological}.

\begin{proof}[Proof of \cref{claim:pathological}]
For each $j\in[0,k_i]$, let $f_j$ denote the fraction of $J\in\cH_i$ with $|J\cap S|=j$.
Then, the left-hand side of~\eqref{eq:claim} is
\begin{equation}\label{eq:LHSa}
\sum_{j=k}^{k_i} f_j |\cH_i|\frac{\binom{m-k_i-(k_i-j)}{w'-(k_i-j)}}{\binom{m-k_i}{w'}}=\sum_{j=k}^{k_i} f_j |\cH_i|\frac{(w')_{k_i-j}}{(m-k_i)_{k_i-j}}.
\end{equation}
Hence, it suffices to show that, for each $j\in[k,k_i]$,
\begin{equation}\label{eq:fraction_compare}
f_j \frac{(w')_{k_i-j}}{(m-k_i)_{k_i-j}}\frac{\binom{m}{k_i}}{\binom{w'+k_i}{k_i}}=e^{O(j)} C^{-j},
\end{equation}
where the implied constants in the $O$ notation are independent of $C$\COMMENT{To see why this suffices, note that we simply need to consider the sum over all possible values of $j$, that is,
\[\sum_{j=k}^{k_i}e^{O(j)}C^{-j}\leq C^{-2k/3}.\]
To check that this holds, simply note that, by adjusting the value of $C$, we may assume that, say, $e^{O(j)}C^{-j}\leq C^{-9j/10}$ for all $j$ in the desired range.
Now this becomes a geometric series, where the rate is less than $1/2$, so the infinite sum (which is clearly an upper bound for the desired sum) is at most twice its first element.
But $2C^{-9k/10}\leq C^{-2k/3}$ for sufficiently large $k$ (which we may assume).}.

We rewrite the left-hand side of~\eqref{eq:fraction_compare} as follows:
\[
f_j \frac{(w')_{k_i-j}}{(m-k_i)_{k_i-j}}\cdot\frac{(m)_{k_i}}{(w'+k_i)_{k_i}}=f_j \frac{(w')_{k_i-j}}{(w'+k_i)_{k_i-j}}\cdot\frac{(m)_{k_i-j}}{(m-k_i)_{k_i-j}}\cdot\frac{(m-k_i+j)_{j}}{(w'+j)_{j}}.
\]
We next use the bounds $\frac{(w')_{k_i-j}}{(w'+k_i)_{k_i-j}}\le 1$, $\frac{(m)_{k_i-j}}{(m-k_i)_{k_i-j}}\le \exp\left(k_i(k_i-j)/(m-k_i)\right)\le e^{d^2}=O(1)$\COMMENT{We have that
\[\frac{(m)_{k_i-j}}{(m-k_i)_{k_i-j}}=\frac{m}{m-k_i}\frac{m-1}{m-k_i-1}\cdots\frac{m-k_i+j+1}{m-2k_i+j+1}\leq\left(\frac{m}{m-k_i}\right)^{k_i-j}=\left(1+\frac{k_i}{m-k_i}\right)^{k_i-j}\leq\exp\left(\frac{k_i(k_i-j)}{m-k_i}\right).\]
Finally, note that, by assumption on the maximum degree of $F$, we have that $k_i\leq nd/2$, so we can bound the above expression by 
\[\exp\left(\frac{k_i(k_i-j)}{m-k_i}\right)\leq\exp\left(\frac{k_i^2}{m-k_i}\right)\leq\exp\left(\frac{n^2d^2/4}{n(n-1-d)/2}\right)\leq\exp\left(\frac{n^2d^2/4}{n^2/4}\right)=e^{d^2}\]
(where some inequalities assume $n$ is sufficiently large).\\
It is worth pointing out that this bound is only relevant for $i=0$, as otherwise this is always $o(1)$.} and 
\[
\frac{(m-k_i+j)_{j}}{(w'+j)_{j}}\le \left(\frac{m}{Cq \binom{n}{2}-k_i}\right)^j \overset{k_i=O(k_0),\,q\geq4k_0/(Cn^2)}{\le}  e^{O(j)}C^{-j}q^{-j}
\]
to bound the left-hand side of~\eqref{eq:fraction_compare} from above by
\begin{equation}\label{eq:LHSb}
f_j e^{O(j)}C^{-j}q^{-j}.
\end{equation}

Finally we observe that, since $\cF$ is $q$-spread and $(\cF,\cH_1,\ldots,\cH_i)$ is a fragmentation process with $|\cH_\ell|\ge |\cH_{\ell-1}|/2$ for all $\ell\in [i]$, by \eqref{equa:fragmentationProperty}, for every $I\subseteq M$ we have
\[
|\cH_i\cap \langle I\rangle|\le|\cF\cap \langle I\rangle|\le q^{|I|}|\cF|\le 2^iq^{|I|}|\cH_i|.
\]
Consider now sets $I\subseteq M$ with $|I|=j$.
For $j\in[\delta e(F), k_i]$, since $\delta e(F)\ge \delta k_i$, we have that\COMMENT{Here we are using the fact that, if $I$ contains any element outside of $\cH_i$, then $|\cH_i\cap \langle I\rangle|=0$.
We then have that
\[
f_j\le \binom{k_i}{j} 2^iq^{|I|}\leq2^{k_i}2^iq^j\leq2^{\delta^{-1}j+i}q^j=e^{O_\delta(j)}q^j.
\]}
\begin{equation}\label{equa:fj1}
    f_j\le \binom{k_i}{j} 2^iq^{|I|}=e^{O_\delta(j)}q^j,
\end{equation}
where we have used the fact that $2^i=O(1)$. 
Now recall that $k_0=e(F)$. 
For $j\in[k, \min\{k_i,\delta k_0\}]$, we use the fact that $\cF$ is $(q,\alpha,\delta)$-superspread and, thus, by \eqref{equa:fragmentationProperty}, for each $I\subseteq M$ with $|I|=j$ and $c_I$ components we have
\[
|\cH_i\cap \langle I\rangle|\le|\cF\cap \langle I\rangle|\le q^{|I|} k_0^{-\alpha c_I}|\cF|\le 2^i k_0^{-\alpha c_I}q^{|I|}|\cH_i|.
\]
This yields
\begin{align}\label{equa:fj2}
f_j&=|\cH_i|^{-1} \left|\{J\in\cH_i:|J\cap S|=j\}\right|
\le  |\cH_i|^{-1} \sum_{I\subseteq S, |I|=j} \left|\cH_i \cap \langle I\rangle\right|
\le |\cH_i|^{-1} \sum_{I\subseteq S, |I|=j}  2^i k_0^{-\alpha c_I}q^{|I|}|\cH_i|\nonumber\\
&= O(q^j)\sum_{I\subseteq S, |I|=j}   k_0^{-\alpha c_I}
= O(q^j) \sum_{c=1}^j  k_0^{-\alpha c}\left|\left\{I\subseteq S: |I|=j, c_I=c\right\}\right|
{\overset{\text{\cref{lem:num_subgraphs}}}{\le}} e^{O(j)}q^j  \sum_{c=1}^j  k_0^{-\alpha c}\binom{k_i}{c}\nonumber\\
&\le e^{O(j)}q^j  \sum_{c=1}^j  \left(\frac{e\cdot k_i\cdot k_0^{-\alpha}}{c}\right)^c
=e^{O(j)}q^j  \sum_{c=1}^j  \left(\frac{e\cdot k}{c}\right)^c\le 
e^{O(j)}q^j  \sum_{c=1}^j  e^{k}\overset{j\ge k}{=}e^{O(j)}q^j.
\end{align}

Combining \eqref{equa:fj1} and \eqref{equa:fj2}, we simplify~\eqref{eq:LHSb} as follows:
\[
f_j e^{O(j)}C^{-j}q^{-j}\le  e^{O(j)}q^je^{O(j)}C^{-j}q^{-j}=e^{O(j)}C^{-j}.
\]
This confirms~\eqref{eq:fraction_compare} and, thus, concludes the proof of the claim.
\end{proof}

\end{document}